\documentclass[a4paper,reqno,oneside,10pt]{amsart} 

\usepackage{epsfig}
\usepackage{graphicx}
\usepackage{psfrag}
\usepackage{pinlabel}
\usepackage{mathrsfs}
\usepackage{amsmath}
\usepackage{amssymb}
\usepackage{extarrows}
\usepackage{enumerate}
\usepackage{mathrsfs}

\newcommand{\C}{\mathbf{C}}
\newcommand{\R}{\mathbf{R}}

\newcommand{\Z}{\mathbf{Z}}

\newcommand{\Sph}{\mathbf{S}}

\newcommand{\Hyp}{\mathbf{H}}

\def\a{\alpha}
\def\b{\beta}
\def\g{\gamma}
\def\d{\delta}
\def\e{\varepsilon}
\def\l{\lambda}
\def\f{\varphi}

\def\s{\sigma}

\newcommand{\teich}{\mathcal{T}}

\newcommand{\ml}{\mathcal{ML}}

\newcommand{\Hold}{\mathcal{H}}

\newcommand{\Isom}{\mathsf{Isom}}
\newcommand{\PSL}{\mathsf{PSL}}
\newcommand{\SL}{\mathsf{SL}}

\newcommand{\Pl}{\mathcal{P}}

\newcommand{\mrm}{\mathrm}
\newcommand{\mc}{\mathcal}

\theoremstyle{plain}
\newtheorem{theorem}{Theorem}[section]

\newtheorem{proposition}[theorem]{Proposition}

\newtheorem{lemma}[theorem]{Lemma}

\theoremstyle{definition}

\newtheorem{question}{Question}[section]

\newtheorem{remark}{Remark}[section]
\newtheorem*{remarknonumber}{Remark}

\newtheorem*{remarksnonumber}{Remarks}

\newtheorem{consequence}{Consequence}[section]

\title[]{Derivatives of length functions and shearing coordinates on Teichm\"uller spaces}
\thanks{This work has been fully supported by FIRB 2010 (RBFR10GHHH$_{003}$).}

\begin{document}
\maketitle
 \begin{small}
\begin{flushleft} 
 \textbf{Matthieu Gendulphe}\\
 Dipartimento di Matematica Guido Castelnuovo\\
 Sapienza Universit\`a di Roma\\
 Piazzale Aldo Moro 5\\
 00185 Roma\\
 email : matthieu@gendulphe.com
\end{flushleft}\end{small}
\medskip

\begin{abstract}
Let $S$ be a closed oriented surface of genus at least $2$, and denote by $\teich(S)$ its Teichm\"uller space. For any isotopy class of closed curves $\g$, we compute the first three derivatives of the length function $\ell_\g:\teich(S)\rightarrow\R_+$ in the shearing coordinates associated to a maximal geodesic lamination $\l$. We show that if $\g$ intersects each leaf of $\l$, then the Hessian of $\ell_\g$ is positive-definite. 
We extend this result to length functions of measured laminations. We also provide a method to compute higher derivatives of length functions of geodesics. We use Bonahon's theory of transverse H\"older distributions and shearing coordinates.
\end{abstract}

\renewcommand{\abstractname}{R\'esum\'e} 
\begin{abstract}
Soit $S$ une surface ferm\'ee orientable de type hyperbolique, et soit $\teich(S)$ son espace de Teichm\"uller. \'Etant donn\'ee une classe d'homotopie libre de courbes ferm\'ees $\g$, nous calculons les trois premi\`eres d\'eriv\'ees de la fonction longueur $\ell_\g:\teich(S)\rightarrow \R_+$ dans les coordonn\'ees \emph{shearing} associ\'ees \`a une lamination g\'eod\'esique maximale $\l$ de $S$. Nous montrons que si $\g$ intersecte toutes les feuilles de $\l$, alors le hessien de $\ell_\g$ est d\'efini-positif. Nous \'etendons ce r\'esultat aux fonctions longueur des laminations mesur\'ees. Nous donnons aussi une m\'ethode pour calculer les d\'eriv\'ees d'ordre sup\'erieur de $\ell_\g$. Nos preuves utilisent les travaux de Bonahon sur les coordonn\'ees \emph{shearing}. 
\end{abstract}
\vspace{1cm}

\setcounter{tocdepth}{1}   
\tableofcontents

\newpage
\section{Introduction}

\subsection*{Statement of the results}
 Let $S$ be a  closed connected oriented surface with negative Euler characteristic $\chi(S)<0$. Its \emph{Teichm\"uller space} $\teich(S)$ is the space of hyperbolic metrics on $S$ up to isotopy. It is a smooth manifold of dimension $-3\chi(S)$. Given a non trivial isotopy class of closed curves $\g$, there is a smooth \emph{length function} $\ell_\g:\teich(S)\rightarrow\R_+^\ast$ that associates to a point $[m]$ in $\teich(S)$ the length $\ell_\g(m)$ of the unique $m$-geodesic in $\g$. These length functions play a crucial role in low-dimensional topology and geometry.\par

 In this article, we study the derivatives of length functions in the shearing coordinates. Given a maximal geodesic lamination $\l$ on $S$, Bonahon (\cite{bonahon-toulouse}) realized $\teich(S)$ as an open convex cone $\mc C(\l)$ in the linear space $\Hold(\l;\R)$ of transverse H\"older distributions for $\l$ . The linear structure of $\Hold(\l;\R)$ is meaningful in terms of hyperbolic geometry. If $\mu$ is a transverse measure for $\l$, the trajectories of the earthquake flow directed by $\mu$ are affine lines in the shearing coordinates. Thus, the linear flow in $\mc C(\l)$ can be seen as  a generalization of the earthquake flow.\par

 Our main theorem gives explicit formulas for the first and second derivatives of length functions in the shearing coordinates:
  
\begin{theorem}
Let $[m]$ be a point in $\teich(S)$, $\l$ be a maximal geodesic lamination, and $\g$ be a closed geodesic transverse to $\l$.
For any transverse H\"older distribution $\a\in \Hold(\l;\R)$, we have
\begin{eqnarray*}
(\mrm d \ell_{\g})_{[m]}(\a)  & = & \int_\g \cos \theta_p\  \mrm d \a(p), \\
(\mrm d^2 \ell_{\g})_{[m]}(\a,\a)  & = & \frac{1}{2\sinh\frac{\ell_\g}{2}} \int_{\g}\int_\g \cosh\left(\frac{\ell_\g}{2}-\ell_{pq}\right) \sin \theta_p \sin \theta_q       \ \mrm d \a(p)\mrm d \a(q),
\end{eqnarray*}
where we denote by $\theta_p$ the $m$-angle at $p$ from $\g$ to the leaf of $\l$ passing through $p$ following the orientation of $S$, and by $\ell_{pq}$ the $m$-length of any of the two segments of $\g$ bounded by $p$ and $q$.
\end{theorem}

\begin{remarksnonumber}\begin{enumerate}
\item We identify the tangent space $T_{[m]}\teich(S)$ with $\Hold(\l;\R)$.
\item The support of $\a$ on $\g$ is contained in $\g\cap\l$, so that the formulas above make sense.
\item When $\g$ is a closed leaf of $\l$, then its length is given by the Theorem~E of \cite{bonahon-toulouse}.
\end{enumerate}
\end{remarksnonumber}

 These are generalizations of classical formulas for the first and second derivatives of length functions along earthquake deformations (Kerckhoff \cite{kerckhoff}, Wolpert \cite{wolpert-derivative,wolpert-symplectic}). Note that a generic geodesic lamination admits only one transverse measure (up to multiplication by a positive scalar), therefore the formulas of Kerckhoff and Wolpert give (generically) the derivatives of length functions for only one direction in the shearing coordinates.\par

We also provide a method to compute recursively the higher derivatives of $\ell_\g$, and we give an explicit formula for the third derivative (\textsection \ref{sec:proof}). It seems possible to find a closed formula for all derivatives using our method. Recently, Bridgeman (\cite{bridgeman}) gave a closed formula for all derivatives of $\cosh(\ell_\g/2)$ along twist deformations. Our method of computation is different in that we use Jacobi fields instead of product of matrices in $\SL(2,\C)$.\par 

We extend the formula for the first derivative to length functions of measured laminations (Theorem~\ref{thm:laminations}), and we show that: 

\begin{theorem}\label{thm:2}
If $(\g,\mu)$ is a measured lamination that intersects each leaf of $\l$, then the Hessian of $\ell_{(\g,\mu)}$ in the shearing coordinates is everywhere positive-definite.
\end{theorem}
 
\begin{remarknonumber}
This is not an obvious consequence of the previous theorem, for positive transverse H\"older distributions are exactly transverse measures (Bonahon \cite[Proposition~18]{bonahon-topology}).
\end{remarknonumber} 

Actually, we give an effective lower bound for the Hessian of $\ell_{(\g,\mu)}$ (Proposition~\ref{lem:hessian}). It was already known that, under the hypothesis of the theorem, the function $\ell_{(\g,\mu)}$ is strictly convex in the shearing coordinates (Bestvina-Bromberg-Fujiwara-Souto \cite{bromberg} for the case of closed geodesics, and Th\'eret \cite{theret} for the case of measured laminations).\par
 
 As well-known, the Teichm\"uller space admits a noncomplete metric of negative sectional curvature, called the \emph{Weil-Petersson} metric. Recently, Wolf (\cite{wolf}) found an explicit formula for the Weil-Petersson Hessian of a geodesic length function, which was already known to be positive-definite (Wolpert \cite[Theorem~4.6]{wolpert-jdg}). As observed by G. Mondello, one part of Wolf's formula is very similar to our formula. This could be explained as follows: this part of Wolf's formula comes from the second variation of the geodesic seen as a curve on $S$ (\cite[\textsection2]{wolf}), and the idea of our computations (\textsection \ref{sec:computations}) is precisely to move the endpoints of a geodesic arc on a fixed hyperbolic surface.\par
 
  In  this article, we only deal with the closed oriented surface $S$. However, our results extend to any compact surfaces with negative Euler characteristic by considering a double cover which is closed and orientable. 
    
\subsection*{Ideas and heuristic}
 There are different presentations of the shearing coordinates (see for instance \cite{bromberg}). We use the one given by Bonahon (\cite{bonahon-toulouse}), where the shearing coordinates of a hyperbolic metric are encoded by a transverse H\"older distribution. This is a crucial point as our proofs are based on the fact that \emph{a transverse H\"older distribution $\a\in\Hold(\l;\R)$ is locally approximated by a sequence $(\a_n)_n$ of linear combinations of Dirac measures}. Here \emph{locally} means that this approximation works only for the restriction of $\a$ to a given arc. In particular, it is not an approximation in $\Hold(\l;\R)$. Our main theorem comes quickly once we have clearly stated this approximation (Lemma~\ref{lem:convergence}).\par
 
 Let us give some heuristic proof of the main theorem. We identify the geodesic $\g$ with an element of $\pi_1(S)$ still denoted $\g$. Let $(A,m_A)$ be the hyperbolic annulus which is the cover of $(S,m)$ with respect to the subgroup $\langle\g\rangle$ of $\pi_1(S)$. The leaves of $\l$ that intersect $\g$ lift to a lamination $\l_A$ of $(A,m_A)$. Working in the annulus $(A,m_A)$ has two advantages:
 \begin{enumerate}
 \item in the space $\Hold(\l_A;\R)$ the H\"older distribution which is the restriction of $\a$ to $\g$ can be approximated by a sequence $(\a_n)_n$ of linear combinations of Dirac measures. So, for any H\"older function $f$ on $\g$, we have $\int_\g f \mrm d\a_n\rightarrow\int_\g f\mrm d \a$ as $n$ tends to infinity.
 \item there is a bijection between the leaves of $\l_A$ and the intersection points of $\l$ with $\g$.
\end{enumerate}
The H\"older distribution $\a_n$ can be written $\a_n=a_1\delta_{l_1}+\ldots +a_n\delta_{l_n}$ where the $l_i$'s are some leaves of the lamination $\l_A$. We associate to $\a_n$ a deformation $t\mapsto m^{t\a_n}_A$ of $m_A$ obtained by shearing along each $l_i$ by an amount equal to $ta_i$. The derivative of $\ell_\g$ along the this deformation is given by $\left[\frac{\mrm d}{\mrm d t} \ell_\g(m^{t\a_n}_A)\right]_t=\int_\g \left[\cos\theta_p\right]_{m^{t\a_n}_A} \mrm d\a_n(p)$, where the bracket means that the angle is evaluated with respect to the metric $m_A^{t\a_n}$. From the work of Bonahon, we find that $\left[\frac{\mrm d}{\mrm d t} \ell_\g(m^{t\a_n}_A)\right]_t$ converges uniformly to $\left[\frac{\mrm d}{\mrm d t} \ell_\g(m^{t\a}_A)\right]_t$ as $n$ tends to infinity,  where $t\mapsto m_A^{t\a}$ is the deformation defined by $\a$. Using the convergence of $(\a_n)_n$ towards $\a$, we conclude that  $\left[\frac{\mrm d}{\mrm d t} \ell_\g(m^{t\a}_A)\right]_0=\int_\g\cos\theta_p\mrm d\a$.\par
  
  In the rest of the paper, following Bonahon (\cite{bonahon-toulouse}), we work in the universal cover $(\tilde S,\tilde m)$ instead of the annulus $(A,m_A)$ used in \cite{bromberg}. Note that $\Hold(\l_A;\R)$ and $\teich(A)$ have infinite dimension, and consequently are not well-adapted to differential calculus.

\subsection*{Organization of the paper} 
 We first take some pages to recall Bonahon's theory of transverse H\"older distributions and shearing coordinates (\textsection\ref{sec:distributions}). Then, we define the sequence $(\a_n)_n$, and show its convergence towards $\a$ (\textsection \ref{sec:convergence}). This enables us to compute the derivatives of the length functions in the shearing coordinates, in particular we prove the main theorem (\textsection\ref{sec:proof}). In \textsection\ref{sec:laminations} we consider the possible extensions of these results to length functions of measured laminations. Finally, we show the positivity of the Hessian of the length functions (\textsection\ref{sec:hessian}). We postpone in \textsection\ref{sec:computations} the computations of the derivatives of some geometric quantities along twist deformations.   

\subsection*{Acknowlegments} I would like to thank Andrea Sambusetti for his support. 
 
\section{Transverse H\"older distributions and shearing coordinates}\label{sec:distributions}

Here we introduce transverse H\"older distributions and shearing coordinates following the work of Bonahon (\cite{bonahon-toulouse,bonahon-topology}, see also \cite[\textsection 1]{bonahon-park}). We closely follow his notations, and sometimes his text.

\subsection*{Geodesic laminations}
Given a hyperbolic metrics $m'$ on $S$, the identity map between $(\tilde S,\tilde m)$ and $(\tilde S,\tilde m')$ is a quasi-isometry, which extends to an equivariant homeomorphism between the visual compactifications. As a consequence, the \emph{boundary at infinity} $\tilde S_\infty$ is a purely topological object, and so is the \emph{space of geodesics} of $\tilde S$ defined by
$$G(\tilde S)=(\tilde S_\infty\times\tilde S_\infty-\Delta)/\Z_2 ,$$
where $\Delta$ is the diagonal, and $\Z_2$ acts by permuting the factors. We transparently identify a geodesic with its unordered pair of limit points.\par

  A \emph{geodesic lamination} $\l$ on $S$ is a collection of disjoint simple $m$-geodesics whose union is closed in $S$. Its lift $\tilde \l$ is a $\pi_1(S)$-invariant closed subset of $G(\tilde S)$.

\subsection*{Transverse H\"older distributions}
Let $(X,d)$ be a metric space. A function $f:X\rightarrow\R$ is \emph{H\"older continuous} if there exists some constants $A>0$ and $1\geq \nu>0$ such that:
$$|f(x)-f(y)|\leq A\  d(x,y)^\nu\quad \forall x,y\in X.$$
We denote by $H_\nu(X;K)$ the linear space of H\"older continuous functions with exponent $1\geq \nu>0$, and support contained in a compact subset $K\subset X$. We equip $H_\nu(X;K)$  with the norm $\|\cdot\|_\nu$ defined by:
\begin{eqnarray*}
\| f \|_\nu & = & \|f\|_\infty + \sup_{x,y\in X} \frac{|f(x)-f(y)|}{d(x,y)^\nu}.
\end{eqnarray*}
There is a continuous injection $H_\nu(X;K)\rightarrow H_{\nu'}(X;K')$ for any $\nu\geq \nu'$, and any  $K\subset K'$.
We denote by $H(X)$ the space of H\"older continuous functions on $X$ with compact support, it is the union of all $H_\nu(X;K)$ with $1 \geq\nu>0$ and $K\subset X$ compact. A \emph{H\"older distribution} on $X$ is a linear form on $H(X)$ whose restriction to each $H_\nu(X;K)$ is continuous. A (positive) Radon measure is a very particular example of H\"older distribution.\par

Let $\l$ be a geodesic lamination on $(S,m)$. A \emph{transverse H\"older distributions} $\a$ for $\l$ is a H\"older distribution on each smooth arc transverse to $\l$ such that: if $H$ is a H\"older bicontinuous homotopy preserving $\l$ between two transverse arcs $k$ and $k'$, then $H$ transports the H\"older distribution on $k$ to the one on $k'$. This invariance property implies that the support of the H\"older distribution on $k$ is contained in $\l\cap k$. A transverse H\"older distribution which is a Radon measure on each transverse arc is called a \emph{transverse measure}.\par

 We denote by $\Hold(\l;\R)$ the linear space of transverse H\"older distributions for $\l$. It is finite dimensional, we precisely have $\dim\Hold(\l;\R)=\dim\teich(S)=-3\chi(S)$ when $\l$ is maximal. By comparison, the dimension of the cone of transverse measures for $\l$ does not exceed $-\frac{3}{2}\chi(S)$. \par

\subsection*{Geodesic H\"older currents}
Transverse  H\"older distributions and transverse measures are of topological nature.
To see that they do not depend on the metric $m$, we interpret them as objects defined on the space of geodesics $G(\tilde S)$. This space has a canonical H\"older structure. The $\tilde m$-angle at a point $x$ in $\tilde S$ defines a metric on $\tilde S_\infty$, and consequently on $G(\tilde S)$. A different choice for $m$ or $x$ would change the metric on $G(\tilde S)$, but not its H\"older class.\par

 A \emph{geodesic H\"older current} (resp. a \emph{geodesic current}) on $S$ is a $\pi_1(S)$-invariant H\"older distribution (resp. Radon measure) on $G(\tilde S)$. Given a geodesic lamination $\l$, transverse H\"older distributions (resp. transverse measures) for $\l$ are in one-to-one correspondance with geodesic H\"older currents (resp. geodesic currents) whose support is contained in $\tilde \l$.
 
\subsection*{Shearing coordinates}
Let us fix a maximal geodesic lamination $\l$. The complement of $\tilde \l$ in $(\tilde S,\tilde m)$ is a disjoint union of ideal triangles. The way these ideal triangles are glued together is encoded by a transverse H\"older distribution $\s_m\in\Hold(\l;\R)$, which depends only on the isotopy class of $m$. The correspondence $m\mapsto \s_m$ induces a real analytic homeomorphism $$\Sigma:\teich(S)\longrightarrow \mc C(\l),$$
whose image $\mc C(\l)$ is an open convex cone in $\Hold(\l;\R)$ bounded by finitely many faces. The transverse measures for $\l$ belong to the boundary of $\mc C(\l)$.  \par

 When defining the shearing coordinates (\cite[\textsection 2]{bonahon-toulouse}), Bonahon uses \emph{transverse cocycles} instead of transverse H\"older distributions. These two concepts are equivalent (\cite[\textsection 6]{bonahon-topology}).

\subsection*{Shearing hyperbolic metrics}
 Given $\a\in\Hold(\l;\R)$, the most natural deformation of $\s_m$with tangent direction $\a$ is $t\mapsto \s_m+t\a$. When $\a$ is a transverse measure for $\l$, this deformation remains in $\mc C(\l)$ for all $t\geq0$, and its pull-back $t\mapsto \Sigma^{-1}(\s_m+t\a)$ to $\teich(S)$ coincide with the trajectory of the earthquake flow passing through $[m]$ and pointing in the direction of $\a$. Now, we assume $\a$ sufficiently small, and we explain how to construct a hyperbolic metric $m^\a$ such that 
\begin{eqnarray*}
[m^\a] & = &  \Sigma^{-1}(\s_m+\a).
\end{eqnarray*}\par

 First, we deform the action by Deck transformations of $\pi_1(S)$ on $\tilde S$ into a discrete and faithful action $\rho^\a:\pi_1(S)\rightarrow \Isom(\tilde S)$. The quotient $(S',m')=(\tilde S,\tilde m)/\rho^\a$ is a hyperbolic surface. Then, we define $m^\a$ as the pull-back of $m'$ by any diffeomorphism $f:S\rightarrow S'$ whose induced homomorphism on fundamental groups is $\rho^\a$. The isotopy class $[m^\a]\in\teich(S)$ does not depend on $f$. We take few lines to recall the construction of $\rho^\a$.\par

 Given two components $P$ and $Q$ of $\tilde S-\tilde \l$, we denote by $\Pl_{PQ}$ the set of components of $\tilde S-\tilde\l$ that separate $P$ from $Q$. For any $R\in \Pl_{PQ}$, we call $g_R^P$ (resp. $g_R^Q$) the boundary geodesic of $R$ in the direction of $P$ (resp. of $Q$). We also set $\a(P,R)=\a(\mathbf{1}_k)$ where $k$ is any geodesic arc in $(\tilde S,\tilde m)$ with endpoints in $P$ and $R$. \par

To each finite subset $\Pl\subset \Pl_{PQ}$, we associate the isometry $\f^\a_\Pl\in\Isom(\tilde S)$ defined by:
\begin{eqnarray*}
\f^\a_{\Pl} & = & T^{\a(P,P_1)}_{g_1^P} T^{-\a(P,P_1)}_{g_1^Q}\ T^{\a(P,P_2)}_{g_2^P}T^{-\a(P,P_2)}_{g_1^Q}\ldots  T^{\a(P,P_n)}_{g_n^P}T^{-\a(P,P_n)}_{g_n^Q} T^{\a(P,Q)}_{g_Q^P},
\end{eqnarray*}
where the $\Pl_i$'s are the elements of $\Pl$ indexed as one goes from $P$ to $Q$, and $T^u_g$ is the $\tilde m$-isometry that translates by a signed length $u$ on the geodesic $g$. We assume that the geodesics that separate $P$ and $Q$ are oriented from right to left seen from $P$.  The sequence $(\f^\a_\Pl)_{\Pl}$ converges in $\Isom(\tilde S)$, we call $\f^\a_{PQ}$ its limit (\cite[\textsection 5]{bonahon-toulouse}).\par
 
 To define $\rho^\a$, we fix a component $P$ of $\tilde S-\tilde \l$, and we set 
 \begin{eqnarray*}
 \rho^\a(\g)  &= &  \f^\a_{P\g(P)}\circ \g,
 \end{eqnarray*}
for every $\g\in\pi_1(S)$. The conjugacy class of $\rho^\a$ does not depend on the fixed component $P$.

\subsection*{Uniform convergence of derivatives}
In this paragraph, we discuss the convergence of the derivatives of some geometric quantities. Here $\ell(\cdot)$ stands for the translation length.

\begin{lemma}
As $\Pl$ tends to $\Pl_{P\g(P)}$, the derivatives of $\a\mapsto \ell(\f^\a_\Pl\circ\g)$ converge towards the derivatives of $\a\mapsto \ell(\f^\a_{PQ}\circ\g)$ uniformly on some neighborhood of the origin in $\Hold(\l;\R)$.
\end{lemma}

 \begin{consequence} To compute the derivatives of $\a\mapsto \ell_\g(m^\a)$, it suffices to compute the derivatives of $\a\mapsto \ell(\f^\a_{\Pl_n}\circ\g)$, and determine their limits.
\end{consequence}

\begin{proof}[Sketch of proof]  We take a complex path following Bonahon (\cite[\textsection 10]{bonahon-toulouse}). We use the complexifcation $\Hold(\l;\R)\oplus i \Hold(\l;\R/2\pi\Z)$ of $\Hold(\l;\R)$, where $\Hold(\l;\R/2\pi\Z)$ is the space of transverse H\"older distributions with values in $\R/2\pi\Z$. Without entering into details, let us mention that $\mc C(\l)\oplus \Hold(\l;\R/2\pi)$ parametrizes pleated surfaces with pleating locus $\l$  (see \cite[Theorem~31]{bonahon-toulouse} for a precise statement).\par

 We identify $\tilde S$ with $\Hyp^2$. For each $\Pl$, the map
$$ \begin{array}{lrll}
\f_\Pl :& \Hold(\l;\R) & \longrightarrow & \Isom^+(\Hyp^2) \\
 & \a & \longmapsto & \f_\Pl^\a  
\end{array},$$ 
extends to a holomorphic map 
$$ \begin{array}{lrll}
\f_\Pl :& \Hold(\l;\R)\oplus i \Hold(\l;\R/2\pi\Z) & \longrightarrow &  \Isom^+(\Hyp^3)\simeq\PSL(2;\C)  \\
 & \a+i\b & \longmapsto & \f_\Pl^{\a+i\b}  
\end{array}.$$ 
Bonahon showed (see the proof of \cite[Theorem31]{bonahon-toulouse}) that the sequence $(\a+i\b\mapsto\f_\Pl^{\a+i\b})_\Pl$ converges uniformly on a small neighborhood of the origin, and deduce that ${\a+i\b}\mapsto \f^{\a+i\b}_{PQ}$ is holomorphic on this neighborhood. We use Cauchy's formula to conclude.
\end{proof}

 We identify $\tilde S$ with $\Hyp^2$. The same line of arguments gives:

\begin{lemma}\label{lem:endpoint}
Let $z\in \partial \Hyp^2$ be an endpoint of some leaf $l\in\tilde \l$. We denote by $\Pl_l$ the set of components of $\Hyp-\tilde \l$ that separate $P$ from $l$. As $\Pl\subset\Pl_l$ tends to $\Pl_l$, the derivatives of $\a\mapsto\f^\a_{\Pl}(z) $ converge uniformly on a small neighborhood of the origin in $\Hold(\l;\R)$ .
\end{lemma} 
 
\begin{proof}[Sketch of proof] 
If $l$ is an isolated leaf of $\tilde \l$, then it bounds a component $R$ of $\Hyp-\tilde \l$ such that $\Pl_l=\Pl_{PR}$. In that case, the lemma is obvious as $(\a+i\b\mapsto \f^{\a+i\b}_\Pl)_{\Pl\subset\Pl_{l}}$ converges uniformly to $\a+i\b\mapsto \f_{PR}^{\a+i\b}$.\par
 
 If $l$ is not an isolated leaf, then we consider a sequence of endpoints $(z_n)_n$ of isolated leaves that converges monotonically to $z$. We denote by $z_n^{\a+i\b}$ the limit of $( \f^{\a+i\b}_\Pl(z_n))_{\Pl\subset\Pl_{l_n}}$, where $l_n$ is the isolated leaf with endpoint $z_n$. By monotonicity, the sequence $(z_n^\a)_n$ converges for $\a\in \Hold(\l;\R)$ small enough. Moreover, the convergence is uniform on some neighborhood of the origin in $\Hold(\l;\R)$, thanks to Dini's theorem. We endow $\partial \Hyp^3\simeq \Sph^2$ with the round metric compatible with the elliptic elements of $\Isom(\Hyp^3)$. From the definition of $\f_\Pl^{\a+i\b}$, it comes that $d_{\Sph^2}(z^{\a+i\b}_n,z_m^{\a+i\b})\leq d_{\Sph^2}(z^{\a}_n,z_m^{\a})$. We deduce that the sequence $(\a+i\b\mapsto z_n^{\a+i\b})_n$ is uniformly Cauchy on some neighborhood of the origin in $\Hold(\l;\R)\oplus i \Hold(\l;\R/2\pi\Z)$. This implies immediately the uniform convergence of the derivatives, for each $z_n$ is holomorphic. It is then not difficult to conclude.
 \end{proof} 
 
  The interest of this lemma lies in the fact that many geometric quantities can be expressed in terms of points on the boundary. Let $p,q,r,s$ be four points on $\partial \Hyp^2$. If the geodesics $(qr)$ and $(ps)$ intersect with an angle $\theta$, then the cross-ratio $[p,q,r,s]$ is equal to $\cos^2(\theta/2)$. If they do not intersect, then $[p,q,r,s]=-\sinh^2(h/2)$, where $h$ is the distance between the geodesics.

 \subsection*{An example}
The shearing coordinates measure how the components of $S-\l$ are glued together.  We did not define the map $\Sigma$ that associates the transverse H\"older distribution $\s_m$ to the hyperbolic metric $m$, but we explained how to pass from the hyperbolic metric $m$ to the hyperbolic metric $m^\a$ with shearing coordinates $\s_m+\a$. This is done by shifting the components of $\tilde S-\tilde \l$ with respect to each other. In this paragraph, we illustrate this construction by considering the case of a finite maximal geodesic lamination.\par

\begin{figure}[h]
\labellist
\small \hair 2pt
\pinlabel $P_0$ [tl] at 320 80
\pinlabel $P_1$ [tl] at 360 80
\pinlabel $Q_0$ [tl] at 545 80
\pinlabel $Q_1$ [tl] at 495 80
\pinlabel $g_0$ [tl] at 295 150
\pinlabel $g_1$ [tl] at 332 150
\pinlabel $g_2$ [tl] at 372 150
\pinlabel $h_0$ [tl] at 562 150
\pinlabel $h_1$ [tl] at 530 150
\pinlabel $h_2$ [tl] at 487 150
\endlabellist
\centering
\includegraphics[height=4cm]{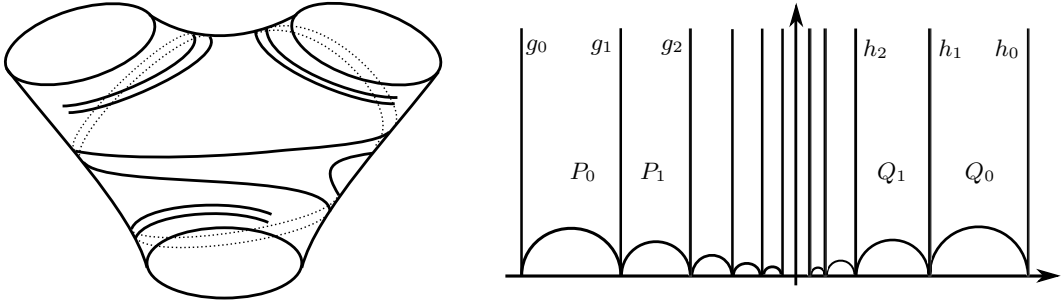}
\caption{A finite maximal lamination}\label{fig:lamination}
\end{figure}

Let $\l$ be a finite maximal geodesic lamination, obtained by adding spiralling geodesics to some pants decomposition (see for instance the picture on the left of Figure~\ref{fig:lamination}). We assume that the geodesics spiral in the same direction along each closed leaf of $\l$. We identify $(\tilde S,\tilde m)$ with the upper-half plane, in such a way that $\tilde \l$ contains a bi-infinite sequence of vertical geodesics converging to the imaginary axis (as on the right of Figure~\ref{fig:lamination}). The imaginary axis is the lift of some closed leaf $\g$ of $\l$, and the other vertical geodesics are lifts of the geodesics spiralling along $\g$. We denote by $g_0,g_1,\ldots$ the vertical geodesics on the left of the imaginary axis, indexed from left to right. We denote by $P_i$ the component of $\tilde S-\tilde\l$ bounded by $g_i$ and $g_{i+1}$. We define similarly $h_i$ and $Q_i$.\par

Given a transverse H\"older distribution $\a$ for $\l$, we observe that
\begin{eqnarray*}
\f^\a_{P_0P_n} &= & T^{\a(P_0,P_1)}_{g_1} T^{\a(P_1,P_2)}_{g_2}\cdots T^{\a(P_{n-1},P_n)}_{g_n}.
\end{eqnarray*} 
Adjacent components are shifted with respect to each other, and the amount of shifting is given by the $\a$-measure of a small geodesic arc intersecting their common boundary. When the $\a$-measure of this arc is positive, the components are shifted to the left with respect to each other. Then, we recognize the classical construction of an earthquake path.\par

 One would expect the sequence of isometries $(\psi^\a_n)_n$ defined by
\begin{eqnarray*}
\psi^\a_{n} & = & T^{\a(P_0,P_1)}_{g_1}\cdots T^{\a(P_{n-1},P_n)}_{g_n} T^{\a(Q_n,Q_{n-1})}_{h_n}\cdots T^{\a(Q_1,Q_0)}_{h_1},
\end{eqnarray*}
to converge towards $\f^\a_{P_0Q_0}$. However, we have $T_{g_i}^{\a(P_{i-1},P_i)}(z)=e^{\a(P_{i-1},P_i)} z+ \ast$, and
\begin{eqnarray*}
\psi_n^\a(z) & = & e^{\a(P_0,P_n)+\a(Q_0,Q_n)}z+\ast,
\end{eqnarray*}
 where $\ast$ replaces any constant term. So, the convergence of $\a(P_0,P_n)+\a(Q_0,Q_n)$ as $n$ tends to infinity is a necessary condition for the convergence of $(\psi^\a_n)_n$. This is automatically satisfied when $\a$ is a transverse measure, but not in general. On the contrary we have
\begin{eqnarray*}
\f_{\Pl}^\a(z) & = & e^{\a(P_0,Q_0)}z+\ast,
\end{eqnarray*}
for $\Pl=\{P_1,\ldots,P_n,Q_n,\ldots,Q_1\}$. This explains Bonahon's choice for the sequence $(\f^\a_\Pl)_\Pl$.\par

 We conclude this paragraph with a remark on the length function of the closed leaf $\g$. The spiralling geodesics divide a pair of pants into two ideal triangles. It implies that $P_2=\g(P_0)$, where $\g\in \pi_1(S)$ translates along the imaginary axis. With some abuse of notations, we use the letter $\g$ for an element of $\pi_1(S)$ whose axis project on the geodesic $\g$. The axis of $\rho^\a(\g)$ is still a vertical geodesic, and its translation length $\ell_\g(m^\a)$ is equal to the logarithm of the ratio between the Euclidean widths of $P_0$ and $\f^\a_{P_0P_2}(P_2)$. So we find that
 \begin{eqnarray*}
 \ell_\g(m^\a) & = & \ell_\g(m)+\a(P_0,P_2).
 \end{eqnarray*}
We also have $ \ell_\g(m^\a)  =  \ell_\g(m)+\a(Q_0,Q_2)$. The quantity $\a(P_0,P_2)$
is the $\a$-mass of some geodesic loop homotopic to $\g$ and transverse to $\l$. The formula above is a particular case of the theorem~E of \cite{bonahon-toulouse}.

\section{Local approximation}\label{sec:convergence}

Once for all, we fix a hyperbolic metric $m$ on $S$, a maximal geodesic lamination $\l$ of $S$, and a transverse H\"older distribution $\a\in\Hold(\l;\R)$.
We consider a geodesic arc $k$ of $(\tilde S,\tilde m)$ whose endpoints belong to two components $P$ and $Q$ of $\tilde S-\tilde \l$. 
We denote by $K\subset G(\tilde S)$ the compact subset containing the leaves of $\tilde \l$ that separate $P$ from $Q$.\par

Let $(\Pl_n)_n$ be an increasing sequence of finite subsets that converges to $\Pl_{PQ}$, we assume for simplicity that $\Pl_n$ is of cardinal $n$.
The isometry $\f_{\Pl_n}^\a$ is encoded by the following H\"older distribution on $G(\tilde S)$:
\begin{eqnarray*}
\a_n & = & \a(P,Q)\ \d_{g^P_Q} + \sum_{i=1}^{n} \a(P_1,P_i) \ (\delta_{g_i^P}-\d_{g_i^Q}).
\end{eqnarray*}
We underline that $\a_n$ has finite support, and is not $\pi_1(S)$-invariant.\par

 We call $\a_K$ the restriction of the geodesic H\"older current $\a$ to $H(K)$. The first lemma below says that $(\a_n)_n$ is a good approximation of $\a_K$. In particular, $\a$ can be approximated \emph{locally} by linear combinations of Dirac measures. This is certainly not possible in $\Hold(\l;\R)$ in general. The second lemma gives an application of the previous result to functions defined by an integral depending on a parameter. \par

\subsection*{Convergence of the sequence}
This lemma is due to Bonahon (\cite[Lemma~9]{bonahon-topology}), but we include a proof as it plays a fundamental role in this article, and as Bonahon's formulation is slightly different.

\begin{lemma}\label{lem:convergence}
The sequence $(\a_n)_n$ converges to $\a_K$ in the topological dual of any $H_\nu(K)$ with $1\geq \nu>0$.
\end{lemma}

\begin{remark}
The proof uses the geometry of the geodesic lamination $\l$ on $(S,m)$.
\end{remark}

\begin{proof}[Proof following Bonahon's ideas]
We consider $\a$ and $\a_n$ as H\"older distributions on $k$ whose supports are included in $k\cap \tilde\l$. It is equivalent to work in $H(K)$ or in $H(k)$ (this is justified by \cite[Lemma~2]{bonahon-topology}, and the proof of the Proposition~5 in \cite{bonahon-topology}).\par
 
We index the elements $P_1,P_2,\ldots,P_n,\ldots$ of $\Pl_{PQ}$ in such a way that the sequence of $\tilde m$-lengths $(\ell(k\cap P_n))_{n}$ is non-increasing. Then, we set $\Pl_n=\{P_1,\ldots,P_n\}$, and we prove the lemma for this particular sequence $(\Pl_n)_n$. This clearly suffices to establish the lemma for any sequence $(\Pl_n)_n$ of finite subsets of $\Pl_{PQ}$ that is increasing and tends to $\Pl_{PQ}$.\par
  
 For any $f\in H_\nu(k)$ we have (\cite[Theorem~2]{bonahon-toulouse}):
\begin{eqnarray*}
 |(\a-\a_n)(f)| & \leq & \sum_{i\geq n} \a(P_1,P_i)\ |f(x_{P_i}^P)-f(x_{P_i}^Q)| ,
\end{eqnarray*}
where $x^P_{P_i}$ and $x_{P_i}^Q$ are the endpoints of $k\cap P_i$. Using the following estimates
$$\left\{\begin{array}{llll}
 d_{\tilde m}(x_{P_i}^-,x_{P_i}^+) & \leq & B e^{-A i} & \textnormal{\cite[Lemma~5]{bonahon-toulouse}} \\
 \a(P_1,P_i)                 & \leq  & C \|\a\| i   & \textnormal{\cite[Lemma~6]{bonahon-toulouse}}
 \end{array}\right.,$$
we find 
\begin{eqnarray*}
  |(\a-\a_n)(f)| & \leq & C\|\a\|\ B^\nu  \|f\|_\nu \left( \sum_{i\geq n} i\ e^{-\nu Ai} \right),
\end{eqnarray*}
where $A,B, C>0$ are independent of $f$, and $\|\cdot\|$ is a norm on $\Hold(\l;\R)$. Clearly, the norm of $\a-\a_N$ in the topological dual of $H_\nu(k)$ tends to zero as $n$ tends to infinity.
\end{proof}

\subsection*{Application to functions defined through transverse H\"older distributions}
Let us recall that we have fixed a point in $\tilde S_\infty$, and that the $\tilde m$-angle at this point induces a Riemannian metric on $G(\tilde S)$. Actually $G(\tilde S)$ equipped with this metric is the interior of a flat M\oe bius band with geodesic boundary.\par

 We denote by $g$ the geodesic of $(\tilde S,\tilde m)$ that supports $k$. The endpoints of $g$ divide $\tilde S_\infty$ into two open intervals, and their product is in bijection with the open subset $U\subset G(\tilde S)$ that contains all geodesics intersecting $g$ transversely. It is easy to find a product of two compact subintervals whose image in $G(\tilde S)$ contains $K$ in its interior. This gives a compact and convex subset $C$ that contains $K$ in its interior.\par

 We consider a smooth function $f:V\times U \rightarrow\R$, where $V\subset \R^m$ is an open subset. We define the functions $F_n:V\rightarrow\R$ and $F:V\rightarrow \R$ by:
$$\left\{\begin{array}{lll}
F_n(p) & = & \int_K f(p,q)\ \mrm d \a_n(q) \\
F(p) & = & \int_K f(p,q)\ \mrm d \a(q)
\end{array}\right.\textnormal{for any }p\in V.$$ 
Note that any function $f(p,\cdot)$ is Lipschitz on $K$ as $f$ is smooth, so that the expressions above make sense. We observe easily that $F_n$ is smooth with
\begin{eqnarray*}
(\mrm d^k F_n)_p(u) & = & \int_K \partial_p^k f(p,\cdot)(u) \ \mrm d \a_n
\end{eqnarray*}
for any $p\in V$ and $u\in (\R^m)^k$.

\begin{lemma}\label{lem:functions}
The function $F$ is smooth, and $(F_n)_n$ converges to $F$ in $C^\infty(V;\R)$, equipped with the topology of uniform convergence of all derivatives on compact subsets. Moreover
\begin{eqnarray*}
(\mrm d^k F)_p(u) & = & \int_K \partial_p^k f(p,\cdot)(u) \ \mrm d \a
\end{eqnarray*}
for any $p\in V$ and $u\in (\R^m)^k$.
\end{lemma}

\begin{proof}
For any $p\in V$, and any $u\in (\R^m)^k$ of unit norm, we have:
\begin{eqnarray*}
|(\mrm d^kF_n)_p(u)-\a_K(\partial^k_p f(p,\cdot)(u))|  & \leq &  \|\a_K- \a_n\|_1  \  \|\partial^k_p f(p,\cdot)(u)\|_1.
\end{eqnarray*}
But, when $p$ belongs to a compact subset $A\subset V$, we have the uniform bound:
\begin{eqnarray*}
\|\partial^k f(p,\cdot)(u )\|_1 & \leq & \sup_{A\times C} \| \partial^k_p f  \|  + \sup_{A\times C} \| \partial_q \partial^k_p f \|,
\end{eqnarray*}
where we use the generic symbol $\|\cdot \|$ for the norms induced by the Riemannian norm on the spaces of multilinear forms on some tangent space. We conclude that $(\mrm d^kF_n)_n$ converges uniformly by applying the previous lemma.
\end{proof}

\section{Computations of the derivatives of length functions}\label{sec:proof}

 We fix a closed geodesic $\g$ that is transverse to $\l$. We are going to compute the first three derivatives of $\ell_\g$ in the shearing coordinates.\par
 
 We consider a geodesic arc $k$ which is a lift of the geodesic $\g$. Following the notations of the previous section, we assume that the endpoints of $k$ belong to two components $P$ and $Q$ of $\tilde S-\tilde \l$, and we denote by $g$ the geodesic of $(\tilde S,\tilde m)$ supporting $k$.\par
 
 In all the formulas below, geometric quantities as $\ell_\g$ or $\cos\theta_p$ have to be evaluated with respect to the fixed metric $m$. 
 
\subsection*{Auxiliary derivatives}
 We recall that $U$ is the open subset of $G(\tilde S)$ that consists in the geodesics intersecting $g$ transversely. We introduce the function $\theta: \teich(S)\times U\rightarrow \R/2\pi\Z$ that associates to a point $[m']\in\teich(S)$, and a geodesic $h\in U$, the oriented $\tilde m'$-angle $\theta_h$ between $g$ and $h$. We assume that $g$ is going from $P$ to $Q$, and that $h$ points to the left when crossed by $g$. The function $\theta$ is clearly smooth. We do not specify the metric in the notation $\theta_h$, because in all the formulas below the angle is measured with respect to $\tilde m$.\par

\begin{lemma}\label{lem:derivatives} Let $\a\in \Hold(\l;\R)$. For any intersection point $p\in \g\cap \l$, we have
\begin{eqnarray*}
(\mrm d \cos\theta_p)_{[m]}(\a) & = & \frac{\sin\theta_p}{2\sinh\frac{\ell_\g}{2}} \ \int_\g \cosh\left(\frac{\ell_\g}{2}-\ell_{pq}\right) \sin \theta_q \ \mrm d \a(q), \\
(\mrm d \sin\theta_p)_{[m]}(\a) & = & -\frac{\cos\theta_p}{2\sinh\frac{\ell_\g}{2}} \ \int_\g \cosh\left(\frac{\ell_\g}{2}-\ell_{pq}\right) \sin \theta_q \ \mrm d \a(q),
\end{eqnarray*}
where $\ell_\g$ is the $m$-length of $\g$. For any pair of intersection points $\{p,q\}\subset \g\cap \l$, we have 
\begin{eqnarray*}
(\mrm d\ell_{pq})_{[m]}(\a) & = &\int_p^q \cos\theta_r \mrm d\a(r) \\
 & & +\frac{1}{2\sinh\frac{\ell_\g}{2}} \int_\g \sinh\left(\frac{\ell_\g}{2}-\ell_{rq}\right) \sin\theta_r\cot\theta_q- \sinh\left(\frac{\ell_\g}{2}-\ell_{rp}\right) \sin\theta_r\cot\theta_p \mrm d\a(r),
\end{eqnarray*}
where, given an orientation for $\g$, we denote by $\ell_{pq}$ the $m$-length of the segment of $\g$ that goes from $p$ to $q$.
\end{lemma}

\begin{remark}
We have $\ell_\g/2-\ell_{pq}=\ell_\g/2-(\ell_\g-\ell_{qp})=-(\ell_\g/2-\ell_{qp})$, so that the first two formulas do not depend on an orientation of $\g$.
\end{remark}

\begin{proof}
The three formulas rely on the same ideas, so we only prove the first one. We work in the universal cover, and we integrate H\"older distributions over the set $K$ defined at the beginning of \textsection\ref{sec:convergence}. This is the same as integrating over $\g$.\par

 As we work in the universal cover, we can consider the variation of a geometric quantity in the direction of a finite approximation $\a_n$. We recall that $\a_n$ is a linear combination of Dirac measures. The deformation associated to $\a_n$ is obtained by shifting along each leaf of $\tilde \l$ that belong to the support of $\a_n$, and the amount of shifting is equal to the corresponding coefficient of the linear combination. For instance, the $\a_n$-displacement of the component $Q$ is given by $t\mapsto \f^{t\a}_{\Pl_n}(Q)$.\par
 
  We denote by $f_n:[-\e,\e]\rightarrow\R$ (resp. $f:[-\e,\e]\rightarrow\R$) the variation of $\cos\theta_p$ in the direction of $\a_n$ (resp. of $\a$). Note that $f(t)$ is equal to $\cos\theta_p$ evaluated with respect to the metric $m^{t\a}$. From the Lemme~\ref{lem:endpoint}, we deduce the pointwise convergence of the derivatives of $f_n$ towards the derivatives of $f$ as $n$ tends to infinity.\par

 We have computed at the end of \textsection\ref{sec:computations} the derivative $f'_n$. It is given by the first formula above with $\a_n$ instead of $\a$, and $K$ instead of $\g$. We conclude by taking the limit of $f'_n(0)$ as $n$ tends to infinity (Lemma~\ref{lem:convergence}).\par 
 
 As regard the other formulas, let us precise that in the computation of the derivative of $\ell_{pq}$, one has to take into account the contribution of the geodesics of $K$ lying between $p$ and $q$. The formula \eqref{eq:distance} works only for the leaves that do not lie between $p$ and $q$. One establishes easily a formula for the other case, using the same ideas as for the proof of the formula~\eqref{eq:distance}.
\end{proof}

\subsection*{The first and second derivatives}
 The last formula of the above lemma implies 
\begin{eqnarray*}
(\mrm d \ell_\g)_{[m]} (\a) & = & \int_\g \cos\theta_p \ \mrm   d \a(p),
\end{eqnarray*}
for all $\a\in\Hold(\l;\R)$. Then, the first formula and the Lemma~\ref{lem:functions} give
\begin{eqnarray*}
(\mrm d^2\ell_\g)_{[m]}(\a,\b) & = & \frac{1}{2\sinh\frac{\ell_\g}{2}} \int_{\g}\int_\g  \cosh\left(\frac{\ell_\g}{2}-\ell_{pq}\right)\ \sin \theta_p \sin \theta_q   \ \mrm d\beta(q)\mrm d \a(p),
\end{eqnarray*} 
for all $\a,\b\in\Hold(\l;\R)$.

\subsection*{Higher derivatives}
The Lemmas~\ref{lem:derivatives} and \ref{lem:functions} enable the computations of all derivatives of $\ell_\g$ in a recursive way. As an example, this is the formula for the third derivative:
\begin{align*}
(\mrm d^{3}\ell_\g)_{[m]}(\a^3) & = -\frac{1}{2\left(\sinh\frac{\ell_\g}{2}\right)^2} \int_\g\int_\g \int_\g &   \cosh(\ell_{pr}-\ell_{rq}) \sin\theta_p\sin\theta_q\cos\theta_r & \\
&&  + \cosh(\ell_{qp}-\ell_{pr}) \cos\theta_p \sin\theta_q \sin\theta_r & \\
&&  + \cosh(\ell_{pq}-\ell_{qr})\sin\theta_p\cos\theta_q\sin\theta_r  & \ \mrm   d\a^3,
\end{align*}
for all $\a\in\Hold(\l;\R)$. This should be compared with Bridgeman's formula for the third derivative of $2\cosh(\ell_\g/2)$ along twist deformations (\cite{bridgeman}). 

\section{Extension to measured laminations}\label{sec:laminations}

In this section, we prove the following generalization of our main theorem:

\begin{theorem}\label{thm:laminations}
Let $\l$ be a maximal geodesic lamination of $(S,m)$, and $(\g,\mu)$ be a measured lamination  transverse to $\l$. For any transverse H\"older distribution $\a\in\Hold(\l;\R)$, we have
\begin{eqnarray*}
(\mrm d \ell_{(\g,\mu)})_{[m]}(\a)  & = & \int \int \cos \theta_p\  \mrm d\mu\ \mrm d \a(p).
\end{eqnarray*}
Moreover, if  $\{(\g_n,\mu_n)\}_n$ is a sequence of weighted simple closed geodesics that converges to $(\g,\mu)$ in the space of measured laminations $\ml(S)$, then the derivatives of $\ell_{(\g_n,\mu_n)}$ converge pointwise towards the derivatives of $\ell_{(\g,\mu)}$ as $n$ tends to infinity: 
\begin{eqnarray*}
\lim_{n\rightarrow\infty}(\mrm d^k \ell_{(\g_n,\mu_n)})_{[m]}(\a_1,\ldots,\a_k) & = & (\mrm d^k \ell_\g)_{[m]} (\a_1,\ldots,\a_k),
\end{eqnarray*}
for any $k\geq 0$ and any $\a_1,\ldots,\a_k\in\Hold(\l;\R)$.
\end{theorem}  

\begin{remark}
To find a formula for the second derivative of $\ell_{(\g,\mu)}$, one has to extend $\ell_{pq}$ to measured laminations, and to study its regularity.
\end{remark}

\begin{question}
It would be interesting to look at the case of geodesic currents.
\end{question}

 In the first paragraph, we give sense to the integral above. In the other paragraphs, we show that the derivatives of $\ell_{(\g_n,\mu_n)}$ converge pointwise towards the derivatives  of $\ell_{(\g,\mu)}$ as $n$ tends to infinity. It is then easy to prove the formula.\par
 
 In what follows, we introduce some sets $U$ and $K_R$ that are different from the sets $U$ and $K$ seen previously. For simplicity, we assume that the geodesic lamination $\g$ is minimal.
   
\subsection*{Definition of the integral}
Let $U$ be the open subset of $G(\tilde S)\times G(\tilde S)$ consisting in couples of intersecting geodesics.
We consider a smooth function $f: \teich(S)\times U\rightarrow \R$, which is invariant with respect to the diagonal action of $\pi_1(S)$ on $U$.
Let us give sense to the following integral:
\begin{eqnarray*}
\int_{l\in \l}\int_{g\in \g}   f(m,g,l)\ \mrm d \mu(g) \mrm d\a(l ) ,
\end{eqnarray*}

We fix an oriented geodesic arc $s$, which is disjoint  from $\l$ and intersects $\g$ transversely. We choose $s$ pointing towards a cusp of $S-\l$. The closure of each component of $\g-s$ is a geodesic arc with endpoints in $s$. We partition these arcs according to their isotopy class relative to $s$. We denote by $\mc R$ the set of such isotopy classes, which is finite since all these arcs are pairwise disjoint.\par

 We fix a class $R\in \mc R$, and we choose two lifts $\tilde s_1$ and $\tilde s_2$ of $s$ to $\tilde S$, such that any geodesic arc in $R$ admits a lift with endpoints in $\tilde s_1$ and $\tilde s_2$. We denote by $K_R\subset G(\tilde S)$ the set of geodesics that intersects $\tilde s_1$ and $\tilde s_2$. We call $h_1$ and $h_2$ the geodesics supporting $\tilde s_1$ and $\tilde s_2$. We denote by $V_R$ the open subset of $G(\tilde S)$ that consists in the geodesics lying between $h_1$ and $h_2$ that intersect any geodesic in $K_R$. We define $F^R:\teich(S)\times V_R\rightarrow \R$ as follows:
\begin{eqnarray*}
F^R(m,l) & = & \int_{\tilde \g\cap K_R} f(m,g,l)\ \mrm d\mu(g).
\end{eqnarray*}
This integral is well defined ($\tilde \g\cap K_R$ is a compact subset of $G(\tilde S)$) and does not depend on the choice of any lift. From a classical theorem of Lebesgue, the function $F^R$ is smooth. Thus $l\mapsto F^R(m,l)$ is H\"older, and $\int_{\tilde \l\cap V^R}  F^R(m,l)\mrm d \a$ is well-defined. We set 
$$
F(m)\ :=\ \int_{l\in \l}  \int_{g\in \g}f(g,l,m)\ \mrm d \mu(g)\ \mrm d\a(l )\  :=\  \sum_{\mc R}\int_{\tilde \l\cap V_R}  F^R(m,l)\ \mrm d \a(l).
$$
The function $F:\teich(S)\rightarrow \R$ is smooth (Lemma\ref{lem:functions}).

\subsection*{From simple closed geodesics to measured laminations}
We consider a sequence $\{(\g_n,\mu_n) \}_n$ of weighted simple closed geodesics that converges to $(\g,\mu)$ in $\ml(S)$. We assume that each geodesic $\g_n$ intersects $s$, and that the arcs made of the components of $\g_n-s$ belong to the classes in $\mc R$. The existence of such a sequence is obvious when working in the train track which has a unique switch correponding to $s$, and one edge for each $R\in \mc R$.\par

 As above, we have a well-defined and smooth function $F_n^R: \teich(S)\times V_R\rightarrow\R$ given by:
$$F^R_n(m,l)=\int_{g\in K_R} f(m,g,l) \mrm d\mu_n(g).$$ 
The restriction of $\mu_n$ to $K_R$ is a linear combination of Dirac measures, so this integral is actually a finite sum. As $\mu_n$ tends to $\mu$ in the space of Radon measures on $G(\tilde S)$, it comes that $F_n^R$ converges to $F^R$ uniformly on compact subsets of $\teich(S)\times V_R$. Similarly, any derivative of the form $(\partial_m^k F_n^R)_{(m,l)}(\a_1,\ldots,\a_k)$ ($k\geq 0$ and $\a_1,\ldots,\a_k\in\Hold(\l;\R)$) converges towards $(\partial_m^k F^R)_{(m,l)}(\a_1,\ldots,\a_k)$ uniformly on compact subsets of $\teich(S)\times V_R$. This implies the $1$-H\"older convergence of all derivatives, for the Lipschitz constant of a smooth function is equal to the supremum of its derivative. We deduce that 
\begin{eqnarray}
(\mrm d^k F)_{[m]}(\a_1,\ldots,\a_k) & = & \lim_{n\rightarrow\infty} (\mrm d^k F_n)_{[m]} (\a_1,\ldots,\a_k),\label{eq:conv-lam}
\end{eqnarray}
for any $k\geq 0$ and any $\a_1,\ldots,\a_k\in\Hold(\l;\R)$.
 
\subsection*{Proof of Theorem~\ref{thm:laminations}}
We take $f$ defined by $f(g,l,m)=\cos\theta(g,l,m)$, where $\theta(g,l,m)$ is the $\tilde m$-angle going from $g$ to $h$ following the orientation of $\tilde S$. According to our main theorem we have $(\mrm d \ell_{(\g_n,\mu_n)})_{[m]}(\a)=F_n(m)$. And, in the paragraph above, we have seen that 
$\lim_{n\rightarrow\infty} F_n(m)\ =\ F(m)$, where the convergence is uniform on compact subsets of $\teich(S)$. As $(\ell_{(\g_n,\mu_n)})_n$ converges pointwise to $\ell_{(\g,\mu)}$, we conclude that
$(\mrm d\ell_{(\g,\mu)})_{[m]}(\a)\ =\ F(m).$\par
We remark that the formula~\eqref{eq:conv-lam} gives the pointwise convergence of all derivatives of $\ell_{(\g_n,\mu_n)}$ towards the derivatives  of $\ell_{(\g,\mu)}$ as $n$ tends to infinity.

\section{Positivity of the Hessian}\label{sec:hessian}

In this section, we give a proof of Theorem~\ref{thm:2}. We interpret 
$$\int_\g\int_\g\cosh\left(\frac{\ell_\g}{2}-\ell_{pq}\right) \sin \theta_p \sin \theta_q    \ \mrm d \a_n^2,$$
as a quadratic form evaluated at $\a_n$, seen as a vector in some finite dimensional linear space. Using an easy lemma of linear algebra, we find an effective lower bound for the integral above. This lower bound increases with $n$. 

\subsection*{A lemma on symmetric matrices with nonnegative entries}
This lemma is certainly well-known, but we were not able to find it in the litterature.
\begin{lemma}
Let $(a_{ij})_{1\leq i,j\leq n}$ be a symmetric matrix satisfying the following conditions:
\begin{enumerate}[i)]
\item $a_{ij}\geq 0$ for all $i,j$ ($A$ is nonnegative),
\item $a_{ii}>a_{ij}$ for all $i,j$.
\end{enumerate}
Then $A$ is definite-positive, that is $x^tAx> 0$ for any $x\in \R^n-\{0\}$.
\end{lemma}

\begin{consequence}\label{cons:linear}
If we replace the condition $ii)$ by the weaker condition $a_{ii}\geq a_{ij}$ ($\forall i, j$), then we find that $A$ is positive, that is $x^tAx\geq 0$ for any $x\in \R^n$. Thus, for any $x\in \R^n$ we have 
$x^tAx  \geq d_1 x_{11}^2+\ldots+ d_nx_{nn}^2$,
 where $d_i=\min\{x_{ii}-x_{ij}~;~ j\neq i\}$.
\end{consequence}

\begin{proof}
We perform the Gauss elimination process to get an upper-triangular matrix which has same leading principal minors as $A$. Let us recall that this process consists in $n$ steps, and that the $k^{\textnormal{th}}$ step gives a matrix $A^{[k+1]}$, obtained from $A^{[k]}$ by doing the operation 
$$L_i-\frac{a^{[k]}_{ik}}{a^{[k]}_{kk}} L_k\rightarrow L_i$$
 for any $i>k$. This operation preserves the leading principal minors. By construction, any diagonal entry $a^{[k]}_{ii}$ of $A^{[k]}$ ($k\geq 2$) remains greater than any other entry of the same line, in particular is positive as $0=a^{[k]}_{i1}$. We conclude that the upper-triangular matrix $A^{[n+1]}$ has positive diagonal entries, thus all its leading principal minors are positive, and $A$ is positive according to Sylvester criterion.
\end{proof}

\subsection*{Positivity of the Hessian}

\begin{lemma}
Let $\g$ be a closed geodesic of $(S,m)$. Let $x_1,\ldots, x_N$ be isolated points of $\g\cap \l$, and denote by $\e_i$ be the minimal distance on $\g$ between $x_i$ and any other point of $\g\cap \l$. For any transverse H\"older distribution $\a\in \Hold(\l;\R)$, we have
\begin{eqnarray*}
(\mrm d^2\ell_\g)_{[m]}(\a^2) & \geq & \sum_{i=1}^N \sinh(\ell_\g/2-\e_i)\ \e_i\  \a(x_i)^2\ \sin^2\theta_{x_i},
\end{eqnarray*}
where $\a(x_i)$ is the $\a$- measure of a small arc of $\g$ that intersects $\l$ only at $x_i$.
\end{lemma}

\begin{proof}
Following the notations of \textsection2 and \textsection 3, we fix a geodesic segment $k\subset \tilde S$ which is a lift of $\g$ whose endpoints belong to two components $P$ and $Q$ of $\tilde S-\tilde \l$. To any intersection point $x_i$ corresponds a unique isolated leaf $p_i$ of $\tilde \l$ that separates $P$ from $Q$.\par

Each $p_i$ is adjacent to two components of $\tilde S-\tilde \l$. For $n$ big enough, the set $\Pl_n$ contains all the components adjacent to the $p_i$'s. Thus, the corresponding H\"older distribution $\a_n$ can be written in the form $\a_n=\sum_{i=1}^n b_i \delta_{p_i}$ with $b_i=\a(x_i)$ for $i\leq N$, and $p_i\neq p_j$ for $i\neq j$. We explicitly have
\begin{eqnarray*}
\int_{k^2} \cosh(\ell_\g/2-\ell_{pq})\sin\theta_p\sin\theta_q\ \mrm d \a_n^2 & = & B^t H B,
\end{eqnarray*}
 where $H$ is the $n\times n$ symmetric matrix given by $H_{ij}=\cosh(\ell_\g/2-\ell_{p_ip_j})$, and $B$ is the vector $B=(b_1\sin\theta_{p_1},\ldots,b_n\sin\theta_{p_n})$.
Let $\e_i$ be the minimal distance on $k$ between $p_i$ and any other leaf of $\tilde \l$. Note that $\e_{i}$ is positive if and only if $p_i$ is an isolated leaf of $\tilde \l$. For any $p_i$ ($1\leq i\leq n$) and any $q\in\tilde \l$ we have 
$$\cosh(\ell_\g/2) -\cosh(\ell_\g/2-\ell_{p_iq})\ \geq\ \cosh(\ell_\g/2) -\cosh(\ell_\g/2-\e_i)\ \geq\ \sinh(\ell_\g/2-\e_i) \e_i.$$
We write $H=H'+D$ where $D$ is the diagonal matrix with $D_{ii}=\sinh(\ell/2-\e_{i})\e_{i}$. According to Consequence~\ref{cons:linear}, the matrix $H'$ is positive and
\begin{eqnarray*}
\int_{k^2} \cosh(\ell_\g/2-\ell_{pq})\sin\theta_{p}\sin\theta_q\ \mrm d \a_n^2 & \geq & \sum_{i=1}^n D_{ii} B_i^2,\\
 & \geq & \sum_{i=1}^N \sinh(\ell_\g/2-\e_{i})\e_{i}\ \a(x_i)^2\ \sin^2\theta_{p_i}.
 \end{eqnarray*}
This proves the lemma as the right-hand side of the inequality does not depend on $n$.
\end{proof}

Using the Theorem~\ref{thm:laminations}, we easily extend the lemma above to measured laminations by considering a sequence of simple closed geodesics converging to $(\g,\mu)$. 

\begin{proposition}\label{lem:hessian}
Let $s_1,\ldots, s_N$ be disjoint geodesic arcs of $(S,m)$ that are supported by isolated leaves of $\l$. For each $i$, we denote by $\e_i>0$ the biggest $\e>0$ such that any geodesic arc of length $\e$ starting at a point in $s_i$ has no other intersection point with $\l$ (except if it is contained in $\l$). Let $(\g,\mu)$ be a measured lamination. For any $\a\in\Hold(\l;\R)$ we have
\begin{eqnarray*}
(\mrm d^2\ell_{(\g,\mu)})_{[m]}(\a^2) & \geq & \sum_{i=1}^N \sinh(\ell_{(\g,\mu)}/2-\e_i)\ \e_i\  \a(s_i)^2\ \int_{s_i} \sin^2\theta_{x} \mrm d\mu(x),
\end{eqnarray*}
where $\a(s_i)$ is the $\a$-measure of a small geodesic arc intersecting $s_i$.
\end{proposition}

\section{Shearing along one geodesic}\label{sec:computations}

 In the hyperbolic plane $\Hyp$, we consider two disjoint non asymptotic geodesics $h$ and $h'$. We fix an isometry $\g$ whose axis $g$ intersect both geodesics, and such that $\g(h)=h'$. The geodesic $g$ has an orientation given by $\g$, and we orient every geodesic intersecting $g$ in such a way that it points to the left when crossed by $g$. Let $\f(t)$ be the isometry of $\Hyp$ that translates by a signed length $t$ on $h'$. We set $\g(t):=\f(t)\circ\g$, and we denote by $g(t)$ the axis of $\g(t)$.

The quotient $\Hyp/\langle \g(t) \rangle$ is a hyperbolic annulus. By construction, both geodesics $h$ and $h'$ project onto the same complete simple geodesic, and the deformation $t\mapsto \Hyp/\langle\g(t) \rangle$ is obtained by shearing along this geodesic.\par

 Given a geodesic $l$ of $\Hyp$ lying between $h$ and $h'$ and intersecting $g$, we want to compute the variations of the intersection point $g(t)\cap l$, and of the angle $\theta_l(t)$ between $g(t)$ and $l$. This enables the computation of the first and second derivatives of the translation length $\ell(t):=\ell(\g(t))$. We stress that $h$, $h'$ and $l$ are fixed, but that the axis $g(t)$ moves with $t$.\par
 
\subsection*{Description of the configuration}
  Let $[a,a']$ be the unique geodesic segment orthogonal to $h$ and $h'$ at its endpoints. Given a geodesic $l$ lying between $h$ and $h'$ and intersecting $g$, we denote by $f_l(t)$ the signed distance on $l$ between $l\cap[a,a']$ and $l\cap g(t)$.\par
  
  As $h'$ is the image of $h$ by any $\g(t)$, we have $\theta_h(t)=\theta_{h'}(t)$ for any $t$. Thus the points $a$ and $a'$ are at the same distance from $g(t)$ on respectively $h$ and $h'$. It follows that $f_{h}(t)=f_{h}(0)-\frac{t}{2}$ and $f_{h'}(t)=f_{h'}(0)+\frac{t}{2}$, in particular
\begin{eqnarray}
 f_{h}'  & \equiv &-  \frac{1}{2}.\label{eq:fh}
\end{eqnarray}
This implies that the midpoint $M$ of $[a,a']$ belongs to $g(t)$ independently of $t$. Therefore $g(t)$ is obtained by rotating $g(0)$ by an angle $\rho(t)$ about $M$. 

\begin{figure}[h]
\labellist
\small\hair 2pt
\pinlabel $h$ at 70 460
\pinlabel $h'$ at 360 460
\pinlabel $l$ at 210 460
\pinlabel $g(t)$ at 260 500
\pinlabel $M$ at 210 570
\pinlabel $a$ at 65 560
\pinlabel $a'$ at 340 562
\pinlabel $\theta_h(t)$ at 95 650
\endlabellist
 \centering\includegraphics[height=4cm]{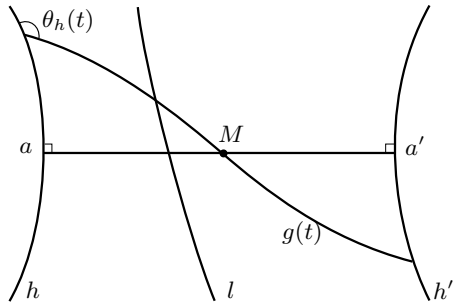}
\caption{Configuration}
\end{figure} 

\subsection*{Variation of the translation length}
The translation length $\ell$ is equal to the distance between the intersection point of $g(t)$ with $h$ and $h'$. The gradient field of the function $d_\Hyp(M,\cdot)$ consists in the unit vectors pointing in the direction opposite to $M$. We deduce easily from \eqref{eq:fh} that
 \begin{eqnarray}
 \ell' & = &  \cos\theta_h .
 \end{eqnarray}
 It is geometrically clear that $\theta_h$ is strictly increasing, so we already find $\ell''\geq0$.
 
\subsection*{Variation of the intersection point}
 We are interested in the trajectory of the point $g(t)\cap l$ on the fixed geodesic $l$.We recall that $g(t)$ is obtained by rotating $g(0)$ by an angle $\rho(t)$ about the midpoint of $[a,a']$. As well-known, if $p(t)$ is a point moving by an angle $\rho(t)$ on a circle of radius $r$ in $\Hyp$, then its velocity $p'(t)$ is equal to $\rho'(t)\sinh(r) $ times the vector tangent to the circle at $p(t)$ pointing in the positive direction. This gives 
\begin{eqnarray*}
 f_l'  & = &\rho'\  \frac{\sinh\left(\frac{\ell}{2}-\ell_{hl}\right)\ }{\sin\theta_l} ,\\
\end{eqnarray*}
where $\ell_{hl}(t)$ is the distance on $g(t)$ from $h$ to $l$. Taking $l=h$ we find
\begin{eqnarray}
 f_{h}'   & = & \rho'\ \frac{\sinh(\frac{\ell}{2})}{\sin\theta_{h}}. \label{eq:fh2}
\end{eqnarray}
Equations \eqref{eq:fh} and \eqref{eq:fh2} enable us to express the angular velocity $\rho'$ in terms of $\ell$ and $\theta_{h}$. We finally obtain:
\begin{eqnarray}
f_l ' & = &-  \frac{1}{2}  \frac{\sin\theta_h}{ \sin\theta_l} \frac{\sinh(\frac{\ell}{2}-\ell_{hl})}{\sinh\frac{\ell}{2}} .
\end{eqnarray}

\subsection*{Variation of the distance between two intersection points}
Let $l'$ be another geodesic of $\Hyp$ lying between $l$ and $h'$, and intersecting $g$. As we know the gradient of the distance function $d_\Hyp(\cdot,\cdot)$, and the variation of each of the intersection points $l\cap g(t)$ and $l'\cap g(t)$, we easily compute 
\begin{eqnarray}
\ell'_{ll'} &= & \frac{\sin\theta_h}{2\sinh\frac{\ell}{2}} \left(\sinh(\frac{\ell}{2}-\ell_{hl'})\cot\theta_{l'} - \sinh(\frac{\ell}{2}-\ell_{hl})\cot\theta_{l} \right)\label{eq:distance}
\end{eqnarray}

\subsection*{Variation of the angle, second variation of the translation length}
We now consider only $l$. We remark that 
\begin{eqnarray*}
\cos\theta_l(t) & = &\pm \langle l'_{p(t)}, Z_{p(t)} \rangle,
\end{eqnarray*}
where $p(t)=l\cap \g(t)$, $l'$ is the unit vector field tangent to $l$, and $Z$ is the unit vector field pointing in the direction of the midpoint $M$ of $[a,a']$. Note that we have $p'(t)=f_l'(t) l'_{p(t)}$. As $l'$ is parallel, we find
\begin{eqnarray*}
\pm\frac{\mrm d \cos\theta_l}{\mrm d t} (t) & = & \langle (\nabla_{f_l' l'}l')_{p(t)}, Z_{p(t)} \rangle + \langle l'_{p(t)}, (\nabla_{f'_ll'}Z)_{p(t)} \rangle \\
& = & f_l' \langle l', (\nabla_{l'} Z)_{p(t)} \rangle 
\end{eqnarray*}
By definition of Jacobi fields we have:
\begin{eqnarray*}
(\nabla_{l'}Z)_{p(t)} &= & -J'(d_\Hyp(M,p(t))),
\end{eqnarray*}
where $s\mapsto J(s)$ is the Jacobi field along the geodesic segment $[M,p(t)]$ satisfying $J(0)=0$ and $J(d(M,p(t))=l'_{p(t)}$. According to the classical formula for the norm of Jacobi fields in hyperbolic spaces we have:
\begin{eqnarray*}
\|J'(s)^\perp \| & = & \sin\theta_{l}\ \frac{\sinh(s)}{\sinh(d_\Hyp(M,p(t)))},
\end{eqnarray*}
where $\perp$ means the component orthogonal to $[M,p(t)]$. We finally find  
\begin{eqnarray*}
\frac{\mrm d \cos\theta_l}{\mrm d t} & = &- f'_l  \ \sin^2\theta_l\  \coth(d_\Hyp(M,g\cap l)), \\
& = & \frac{1}{2} \     \frac{\cosh(\frac{\ell}{2}-\ell_{hl})}{\sinh\frac{\ell}{2}} \        \sin\theta_h \sin\theta_l.
\end{eqnarray*}
This is also equal to the second variation of the translation length because $\ell'=\cos\theta_h$. We have already proved that $\ell''\geq 0$, so there is no sign issue.


\bibliographystyle{alpha}
\bibliography{biblio}

\begin{thebibliography}{BBFS13}

\bibitem[BBFS13]{bromberg}
M.~Bestvina, K.~Bromberg, K.~Fujiwara, and J.~Souto.
\newblock Shearing coordinates and convexity of length functions on
  {T}eichm\"uller space.
\newblock {\em Amer. J. Math.}, 135(6):1449--1476, 2013.

\bibitem[Bon96]{bonahon-toulouse}
F.~Bonahon.
\newblock Shearing hyperbolic surfaces, bending pleated surfaces and
  {T}hurston's symplectic form.
\newblock {\em Ann. Fac. Sci. Toulouse Math. (6)}, 5(2):233--297, 1996.

\bibitem[Bon97]{bonahon-topology}
F.~Bonahon.
\newblock Transverse {H}\"older distributions for geodesic laminations.
\newblock {\em Topology}, 36(1):103--122, 1997.

\bibitem[Bon01]{bonahon-park}
F.~Bonahon.
\newblock Geodesic laminations on surfaces.
\newblock In {\em Laminations and foliations in dynamics, geometry and topology
  ({S}tony {B}rook, {NY}, 1998)}, volume 269 of {\em Contemp. Math.}, pages
  1--37. AMS, 2001.

\bibitem[Bri]{bridgeman}
M.~Bridgeman.
\newblock Higher derivatives of length functions along earthquakes
  deformations.
\newblock {\em Michigan Math. J.}
\newblock To appear.

\bibitem[Ker83]{kerckhoff}
S.~Kerckhoff.
\newblock The {N}ielsen {R}ealization {P}roblem.
\newblock {\em Ann. Math. (2)}, 117(2):235--265, 1983.

\bibitem[Th{\'e}14]{theret}
G.~Th{\'e}ret.
\newblock Convexity of length functions and {T}hurston's shear coordinates.
\newblock Available on the server arXiv, 2014.

\bibitem[Wol81]{wolpert-derivative}
S.~Wolpert.
\newblock An elementary formula for the {F}enchel-{N}ielsen twist.
\newblock {\em Comment. Math. Helv.}, 56(1):132--135, 1981.

\bibitem[Wol83]{wolpert-symplectic}
S.~Wolpert.
\newblock On the symplectic geometry of deformations of a hyperbolic surface.
\newblock {\em Ann. of Math. (2)}, 117(2):207--234, 1983.

\bibitem[Wol87]{wolpert-jdg}
S.~Wolpert.
\newblock Geodesic length functions and the {N}ielsen problem.
\newblock {\em J. Differential Geom.}, 25(2):275--296, 1987.

\bibitem[Wol12]{wolf}
M.~Wolf.
\newblock The {W}eil-{P}etersson {H}essian of length on {T}eichm\"uller space.
\newblock {\em J. Differential Geom.}, 91(1):129--169, 2012.

\end{thebibliography}

\end{document}